\documentclass[10pt]{article}

\usepackage{latexsym,color,amsmath,amsthm,amssymb,amscd,amsfonts}
\usepackage{color}
\usepackage{dsfont}
\usepackage{yfonts}
\usepackage{pst-node}
\usepackage{tikz-cd}

\global\long\def\iprod#1#2{\left\langle #1,\,#2\right\rangle }
\numberwithin{equation}{section}

\theoremstyle{plain}
\newtheorem{thm}{Theorem}[section]
\newtheorem{lem}[thm]{Lemma}
\newtheorem{prop}[thm]{Proposition}

\theoremstyle{definition}

\newtheorem{rmk}{Remark}[section]

\renewcommand{\det}{\mathrm{det}}
\newcommand{\R}{\mathbb R}
\renewcommand{\S}{\mathbb S}

\newcommand{\as}{\mathrm{as}}

\begin{document}

\title{Constrained convex bodies with extremal affine surface areas
}

\date{}

\author{O. Giladi, H. Huang, C. Sch\"utt  and E. M. Werner
		\thanks{Partially supported by  NSF grant DMS-1811146}}

\maketitle

\begin{abstract}
Given a convex body $K\subseteq \R^n$ and $p \in \mathbb{R}$,  we introduce and study the extremal inner and outer affine surface areas
\[ IS_p(K) = \sup_{K'\subseteq K}\big(\as_p(K')\big) \  \text{ and } \   os_p(K)=\inf_{K'\supseteq K}\big(\as_p(K')\big),  \]
where $\as_p(K')$ denotes the $L_p$-affine surface area of $K'$, and the supremum is taken over all convex  subsets of $K$
and the infimum over all convex compact subsets containing $K$. 
\newline
The convex body that realizes $IS_1(K)$ in dimension $2$ was determined in \cite{Barany1997} where it was also shown that
this body is the limit shape of lattice polytopes in $K$.  In higher dimensions no results are known about the extremal bodies.
\par
\noindent
We use a thin shell estimate of \cite{GuedonMilman} and  the L\"owner ellipsoid  to give asymptotic estimates on the size of  $IS_p(K)$
and $os_p(K)$.
Surprisingly, it turns 
out that both quantities are   proportional to a power of volume.
\end{abstract}

\section{Introduction}
\noindent
F. John proved in \cite{John} that  among all ellipsoids contained in  a convex body \( K\in \R^n\), there is a unique ellipsoid of maximal volume, now 
called  the John ellipsoid of \(K\).  Dual to the John ellipsoid is the L\"owner ellipsoid, the ellipsoid of minimal volume containing $K$. These ellipsoids  play  fundamental roles in asymptotic convex geometry.  They are  related to the isotropic
position,  to the study of volume concentration, volume ratio, reverse isoperimetric inequalities, Banach-Mazur distance of normed spaces, and many more, including the 
hyperplane conjecture, one of the major open problems in asymptotic geometric analysis.
We refer to e.g., the books \cite{ArtsteinGiannopoulosMilmanBook, BGVV14} for the details and more information.
\par
\noindent
In this paper, we introduce the analogue to John's theorem, when volume is replaced by affine surface area. 
In parallel to John's maximal volume ellipsoid, respectively   the minimal volume L\"owner ellipsoid, we investigate  these convex bodies  contained in \( K\), respectively containing $K$,  that have the largest, respectively smallest,  $L_p$-affine surface areas, 
\begin{equation}\label{k0}
IS_p(K) = \sup_{K'\subseteq K}\big(\as_p(K')\big)  \  \text{ and } \   os_p(K)=\inf_{K'\supseteq K}\big(\as_p(K')\big).  
\end{equation}
By compactness and continuity, the supremum and infimum are  in fact a maximum and minimum, i.e., $IS_p(K) = as_p(K_0)$ for some convex body $K_0 \subset K$
and $os_p(K) = as_p(K_1)$ for some convex body $K_1 \supset K$.
\par
\noindent
 For $p>1$, 
the $L_p$-affine surface area was introduced by 
E. Lutwak   in his ground breaking paper  \cite{Lutwak96} in the context of the $L_p$-Brunn-Minkowski theory 
and in \cite{SchuettWerner2004} for all other $p$, (see also \cite{Hug, MeyerWerner2000}). $L_1$-affine surface area is classical and goes back to W. Blaschke \cite{Blaschke}.
The definition of $L_p$-affine surface area is given below  in (\ref{def:paffine}), where we also list some of its properties.
Due to its remarkable properties, this notion is  important in many areas of mathematics and applications.
 We only quote characterizations 
of $L_p$-affine surface areas by  M. Ludwig and M. Reitzner \cite{LudwigReitzner2010}, 
the $L_p$-affine isoperimetric inequalities, proved by 
E. Lutwak \cite{Lutwak96} for $p>1$ and for all other $p$  in \cite{WernerYe2008}. The classical case $p=1$ goes back to W. Blaschke \cite{Blaschke}. 
These inequalities  are related to various other inequalities, see e.g., 
E. Lutwak, D. Yang and G. Zhang \cite{Lutwak2000, Lutwak2002}.
In particular, the affine isoperimetric inequality implies the Blaschke-Santal\'o inequality and it proved to be
the key ingredient in the solution of many problems, see e.g.\ the books by R. Gardner \cite{GardnerBook} and R. Schneider \cite{SchneiderBook} and also
\cite{IvakiStancu2013,  Ludwig2010, LutwakOliker1995, SchusterWannerer2012, Stancu2002, TW2, WernerYe2008}.
Recent developments include extensions to an Orlicz theory, e.g.,  \cite{GardnerHugWeil2014, HuangLutwakYangZhang, Ludwig2010,  Ye2015}, to a functional setting  \cite{CFGLSW, CaglarWerner2014} and to the spherical and 
hyperbolic setting \cite{BesauWerner2015, BesauWerner2016}.
\newline
Applications of affine surface areas have been manifold. For instance,
affine surface area appears in best and random approximation of convex bodies by polytopes, see, e.g.,  K. B\"or\"oczky
 \cite{Boeroetzky2000, Boeroetzky2000a}, P. Gruber \cite{Gruber1983, GruberHandbook}, M. Ludwig \cite{Ludwig1999},  M. Reitzner \cite{Reitzner2002, ReitznerSurvey} and also
 \cite{GroteWerner, GroteThaeleWerner, HoehnerSchuettWerner,  Schuett1991, SchuettWerner2003} 
and has connections to, e.g., concentration of volume,  \cite{FleuryGuedonPaouris, Ludwig2010, Lutwak2002}, differential equations \cite{BoeroetzkyLutwakYangZhang, HaberlSchuster2009, HuangLutwakYangZhang, TW2, TrudingerWang2008, Zhao2016}, and 
information theory, e.g.,  \cite{AKSW2012, CaglarWerner2014, LutwakYangZhang2002, LutwakYangZhang2004,
PaourisWerner2012, Werner2012}.
\par
\noindent
In dimension $2$ and for $p=1$, $
IS_1(K)$ was determined exactly  by I. B\'ar\'any  \cite{Barany1997}. Moreover,  he showed in \cite{Barany1997} that the extremal body $K_0$ of (\ref{k0}) is unique and that 
$K_0$ is the limit shape of lattice polygons contained in $K$.  
\par
\noindent
In higher dimensions and for $p\neq 1$, there are no results  available  on $IS_p(K)$, $os_p(K)$ and related notions $OS_p(K)$ and $is_p(K)$, defined in (\ref{def sup}) and (\ref{def inf}) below. We observe  that only certain $p$-ranges are meaningful for the various notions. 
\par
\noindent
We use a thin shell estimate by  Gu\'edon and E. Milman \cite{GuedonMilman}, see also G. Paouris \cite{Paouris} and Y. T. Lee and S. S. Vempala \cite{LeeVempala17}, on concentration 
of volume to 
show in our main theorem that 
 $IS_p(K)$ is proportional to a power of the volume $|K|$ of $K$ for fixed $p$, 
up to a constant depending only on $n$.
It involves the Euclidean unit ball $B^n_2$ centered at $0$, and the isotropic constant 
$L_K^2$ of $K$, defined by 
\begin{equation} \label{LK}
n L_K^2= \min \left\{\frac{1}{|TK|^{1 + \frac{2}{n}}} \int_{a+TK} \|x \| ^2 dx : a \in \mathbb{R}^n, T \in GL(n) \right\}.
\end{equation}
\par
\noindent
{\bf Theorem 3.4.}  
{\em  There is a  constant $C>0$ such that for all $n \in \mathbb{N}$, all $0 \leq p \leq n$ and all convex bodies 
$K\subseteq \R^n$,  
\begin{align*}
\frac{1}{n^{5/6}}  \  \left(\frac{C}{L_K}\right)^{\frac{2np}{n+p}}  \ 
 \frac{IS_p(B^n_2)}{ |B^n_2|^\frac{n-p}{n+p}}\leq \  \frac{IS_p(K)}{ |K|^\frac{n-p}{n+p}} \leq  \   \frac{IS_p(B^n_2)}{ |B^n_2|^\frac{n-p}{n+p}}.
\end{align*}
Equality holds trivially in the right inequality if $p=0, n$. 
If $p \neq 0, n$, equality holds in the right inequality iff $K$ is a centered ellipsoid.}
\vskip 2mm
\noindent
Since 
$
\frac{IS_p(B^n_2)}{ |B^n_2|^\frac{n-p}{n+p}} = n |B^n_2|^ \frac{2p}{n+p},
$
which is asymptotically  equivalent to $\frac{c^\frac{np}{n+p}}{n^\frac{n(p-1)-p}{n+p}}$
  with an absolute constant $c$, the theorem shows that  
for a fixed $p$, $IS_p(K)$ is proportional to 
a power of $|K|$,  up to a constant  depending on $n$ only.
\par
\noindent
We use the L\"owner ellipsoid of $K$ (e.g., \cite{ArtsteinGiannopoulosMilmanBook, BGVV14} or the survey \cite{Henk}), to give 
asymptotic estimates on the size of  $os_p(K)$ and $OS_p(K)$, also in terms of powers of $|K|$,  in Theorem \ref{thm asymp2}. For instance, we show that for $-n < p \leq 0$,
\begin{align*}\label{}
 \frac{os_p(B^n_2)}{ |B^n_2|^\frac{n-p}{n+p}}\leq \  \frac{os_p(K)}{ |K|^\frac{n-p}{n+p}} \leq  \   n^{n \frac{n-p}{n+p}}\  \frac{os_p(B^n_2)}{ |B^n_2|^\frac{n-p}{n+p}}.
\end{align*}
Equality holds trivially in the left inequality if $p=0$. 
If $p \neq  0$, equality holds in the left inequality iff $K$ is a centered ellipsoid. 
\newline
If $K$ is centrally symmetric, $n^{n \frac{n-p}{n+p}}$  can be replaced by $n^{n \frac{n-p}{2(n+p)}}$.
\par
\noindent
We refer to Theorem \ref{thm asymp2} for the details.

\section{Background and definitions}
\noindent
Throughout the paper,  $c, C$ etc., denote  absolute constants that may change from line to line. 
The center of gravity $g(K)$ of $K$ is defined by
$$
g(K)= \frac{1}{|K|} \ \int_{K} x  \ dx.
$$
When the center of gravity of $K$ is at $0$, then,  for real  $p \neq -n$, the $L_p$-affine surface areas are defined as \cite{Lutwak96, MeyerWerner2000, SchuettWerner2004}
\begin{equation} \label{def:paffine}
as_{p}(K)=\int_{\partial K}\frac{\kappa(x)^{\frac{p}{n+p}}}
{\langle x,N(x)\rangle ^{\frac{n(p-1)}{n+p}}} d\mu(x), 
\end{equation}
where  $\kappa(x)$ is the (generalized) Gauss-Kronecker curvature at $x\in \partial K$,  $N(x)$ is the outer unit
normal vector at $x$ to $\partial K$, the boundary of $K$, and  $\langle
\cdot, \cdot \rangle$ is the standard inner product on $\R^n$
which induces the Euclidean norm $\|\cdot\|$.  When the center of gravity of $K$ is not at $0$, we shift $K$ so that it is.
The case $p=1$ is the classical affine surface area whose definition goes back to Blaschke \cite{Blaschke}.
\par
\noindent
We denote by  $\mathcal K_K$  the collection of all compact convex subsets of $K$  and by  
$\mathcal K^K$  the collection of all compact convex sets containing $K$.
\newline
For $-\infty \leq p \leq \infty$, $p \neq -n$, we then define the 
{\em inner and outer maximal affine surface areas}  by
\begin{equation}\label{def sup}
IS_p(K) =  \sup_{C \in \mathcal K_K}\big(\as_p(C)\big),           \ \   \   OS _p(K)= \sup_{C \in \mathcal K^K}\big(\as_p(C)\big),  
\end{equation}
and the {\em inner and outer mininal affine surface areas} by
\begin{eqnarray}\label{def inf}
{i s}_p(K)=  \inf_{C \in \mathcal K_K}\big(\as_p(C)\big),         \ \   \   {os}_p(K)=  \inf_{ C \in \mathcal K^K }\big(\as_p(C)\big).
\end{eqnarray}
We show in section  \ref{Relevant-p} that $is_p$ is identically equal to $0$ for all $p$ and all $K$ and that the only meaningful $p$-range for 
$IS_p$ is $[0,n]$, for $OS_p$ it  is $[n,\infty]$ and for $os_p$ it is $(-n,0]$.
\par
\noindent
By Blaschke's selection theorem,  $\mathcal K_K$   is compact with respect to the Hausdorff metric.   
Proposition \ref{prop continuous} below, proved in  \cite{Lutwak96}, shows that the functional $K\mapsto as_p(K)$ is  upper semicontinuous with respect to the Hausdorff metric, 
if 
$0 \leq p \leq \infty$, respectively lower semicontinuous if $-n < p \leq 0$. We show in Lemma \ref{lemma-extreme} that 
 the suprema in~\eqref{def sup}  are in fact maxima for the relevant $p$-ranges $0 \leq p \leq n$, respectively, $n \leq p \leq \infty$, 
\begin{align*}\label{def sup}
IS_p(K) = \as_p(K_0) \hskip 3mm  \text{and} \hskip 3mm OS_p(K) = \as_p(K_1)
\end{align*}
for some convex body $K_0 \subset K$, respectively $K \subset K_1  $, 
 and  that the second  infimum in~\eqref{def inf}  is in fact a minimum for $-n < p \leq 0$, 
\begin{equation*}
os_p(K) = \as_p(K_2), 
\end{equation*}
for some $K_2$ in $\mathcal K^K$.
\par
\noindent
It was shown  \cite{Lutwak96,  SchuettWerner2004} that  for all $p \neq -n$ and for all invertible linear transformations $T:\R^n \to \R^n$
\begin{equation}\label{pafteraffine}
\as_p(T(K)) = |\det(T)|^{\frac{n-p}{n+p}}\  \as_p(K).
\end{equation}
It then follows immediately from the definitions ~\eqref{def sup}  
and ~\eqref{def inf} that the same holds,  replacing $\as_p$ with $IS_p, OS_p,is_p$ and 
$os_p$.
\par
\noindent
For a general convex body $K$ in $\mathbb{R}^n$, a  
particularly useful way to define $as_1(K)$ is the following. For $u\in \R^n$ and $t \ge 0$, define the half-spaces
\begin{align*}
H^+(t,u) = \big\{x\in \R^n~\big|~ \langle x,u \rangle \ge t\big\}, \quad H^-(t,u) = \big\{x\in \R^n~\big|~ \langle x,u \rangle \le t\big\}.
\end{align*}
For a convex body $K\subseteq \R^n$ and $\delta >0$, the (convex) floating body $K_\delta$ was introduced independently by B\'ar\'any and Larman
\cite{BaranyLarman1988}  and Sch\"utt and Werner \cite{SchuettWerner1990}, 
\begin{equation}\label{FB}
K_{\delta} = \bigcap_{|H^+(t,u) \cap K | \le \delta |K|}H^{-}(t,u).
\end{equation}
It was shown in~\cite{SchuettWerner1990} that for any convex body $K$ in $\mathbb{R}^n$, 
\begin{align} \label{limitFB}
\as_1(K) = 2\left(\frac{|B_2^{n-1}|}{n+1}\right)^{\frac {2}{n+1}}\lim_{\delta \to 0}\frac{|K|-|K_{\delta}|}{\big(\delta |K|\big)^{\frac 2{n+1}}}.
\end{align}
Here, and in what follows, $B_2^n$ denotes the unit Euclidean ball in $\R^n$. 
\vskip 2mm
\noindent
Geometric descriptions in the sense of (\ref{FB}) and  (\ref{limitFB}) of $L_p$-affine surface area also exist. We refer to e.g.,  \cite{HuangSlomkaWerner, SchuettWerner2003, SchuettWerner2004, Werner2002, WernerYe2008}. 
\vskip 5mm
\section{Main results}
Our main results  give quantitative estimates for the   inner and outer extremal affine surface areas. 
We observe  first that for some $p$, the values for the extremal affine surface areas can  can be given 
explicitly and  the $p$-ranges  can be restricted accordingly in the quantitive estimates of Theorems \ref{thm asymp} and \ref{thm asymp2} below.
\par
\noindent
\subsection{The relevant $p$-ranges} \label{Relevant-p}
(i) {\bf The case $IS_p(K)$} \label{IS}
\par
\noindent
If $p=0$, then for all $K$, 
$$
IS_0(K) =  \sup_{K'\in \mathcal K_K}\big(\as_0(K')\big) = n \sup_{K'\in \mathcal K_K} |K'| = n|K|.  
$$
\par
\noindent
If $p=n$, then for all $K$, 
$$
IS_n(K) = n |B^n_2|.
$$
Indeed, on the one hand, we have  by (\ref{pafteraffine}), 
$$
IS_n(K) \geq   \sup_{\rho B^n_2 \in \mathcal K_K}\big(\as_n(\rho B^n_2)\big) =  \sup_{\rho B^n_2 \in \mathcal K_K}\big(\as_n(B^n_2)\big)  = n |B^n_2|.
$$
The equi-affine isoperimetric inequality  \cite{Lutwak96} says that $as_n(K) \leq as_n(B^n_2)$. Therefore, 
$$
IS_n(K) =    \sup_{K'\in \mathcal K_K}\big(\as_n(K')\big) \leq    \sup_{K'\in \mathcal K_K} \big(\as_n( B^n_2)\big) =  n |B^n_2|.
$$
\par
\noindent
If $n < p \leq \infty$, then $IS_p(K) = \infty$.
This holds as by (\ref{pafteraffine}), 
$$
IS_p(K) \geq   \sup_{\varepsilon B^n_2 \in \mathcal K_K}\big(\as_p(\varepsilon B^n_2)\big) =  \sup_{\varepsilon} \varepsilon^{n \frac{n-p}{n+p}} \  n |B^n_2| = \infty, 
$$
since $\frac{n-p}{n+p} <0$.
\par
\noindent
If $- n < p<0$, then for all $K$,  $IS_p(K) = \infty$.
Indeed,  we have for all polytopes $P$
$$
IS_p(K) \geq   \sup_{P \in \mathcal K_K}\big(\as_p(P)\big) = \sup_{P \in \mathcal K_K}
\int_{\partial P}\frac{\kappa(x)^{\frac{p}{n+p}}}
{\langle x,N(x)\rangle ^{\frac{n(p-1)}{n+p}}} d\mu(x) = \infty,
$$
since $\kappa(x)=0$ almost everywhere.
\par
\noindent
If $- \infty \leq  p< -n$,  then for all $K$,  $IS_p(K) = \infty$.
Indeed, as above,
$$
IS_p(K) \geq   \sup_{\varepsilon B^n_2 \in \mathcal K_K}\big(\as_p(\varepsilon B^n_2)\big) =  \sup_{\varepsilon} \varepsilon^{n \frac{n-p}{n+p}} \  n |B^n_2| = \infty, 
$$
since $\frac{n-p}{n+p} <0$.
\par
\noindent
{\bf Conclusion}. The relevant $p$-range for $IS_p$ is $p \in [0,n]$. \\
We  note also that for $p \in [0,n]$, 
\begin{equation}\label{IS-ball}
IS_p(B^n_2)=n |B^n_2| = as_p(B^n_2).
\end{equation}

\vskip 2mm
\noindent
(ii) {\bf The case $OS_p(K)$} \label{OS}.
\par
\noindent
If $p=n$, then for all $K$, $OS_n(K) = n |B^n_2|$. 
Similarly, to (i) above, 
$$
OS_n(K) \geq   \sup_{R B^n_2 \in \mathcal K^K}\big(\as_n(R B^n_2)\big) =  \sup_{R B^n_2 \in \mathcal K^K}\big(\as_n(B^n_2)\big)  = n |B^n_2|
$$
and again by the equi-affine isoperimetric inequality, 
$$
OS_n(K) =    \sup_{K'\in \mathcal K^K}\big(\as_n(K')\big) \leq    \sup_{K'\in \mathcal K^K} \big(\as_n( B^n_2)\big) =  n |B^n_2|.
$$
\par
\noindent
If $0 \leq p < n$, then, $OS_p(K) = \infty$.  This holds as 
$$
OS_p(K) \geq   \sup_{R B^n_2 \in \mathcal K^K}\big(\as_p(R B^n_2)\big) =  \sup_{R B^n_2 \in \mathcal K^K} R^{\frac{n-p}{n+p}}  n |B^n_2|,
$$
and $R$ can be made arbitrarily large.
\par
\noindent
If $- n < p < 0$, then, $OS_p(K) = \infty$.
\newline
This holds as we can again take polytopes $P$ that contain $K$.
\par
\noindent
If $- \infty \leq  p < -n$, then for all $K$,  $OS_p(K) = \infty$.
\newline
Let $C_\varepsilon$ be a rounded cube centered at $0$ containing $K$ and such that each vertex is rounded by replacing it by a 
Euclidean ball with radius $\varepsilon$. More specifically,
$C_{\varepsilon}$ is the convex hull of the $2^{n}$ Euclidean balls
$$
B_{2}^{n}(t\cdot \delta,\epsilon)
\hskip 20mm \delta=(\delta_{1},\dots,\delta_{n})
$$
where $\delta_{i}=\pm1$ for all $i=1,\dots,n$ and $t$ is sufficiently big so that the convex
hull contains $K$. The boundary of $C_{\varepsilon}$ contains all the $2^{n}$-tants of the 
boundary of $B_{2}^{n}$. Therefore, in order to estimate $as_p\left( C_\varepsilon \right)$
from below it suffices to restrict the integration over the boundary of $C_{\varepsilon}$ to those
$2^{n}$-tants of the 
boundary of $B_{2}^{n}$. The curvature there equals $\varepsilon^{-n+1}$, while
$$
\langle x,N(x)\rangle\leq 2t\cdot \sqrt{n}.
$$
Then
$$
OS_p(K) \geq   as_p\left( C_\varepsilon \right)  
\geq \frac{\varepsilon^{ \frac{n(n-1)}{n+p}}}{(2t\sqrt{n})^{n\frac{p-1}{n+p}}}  n |B^n_2|,
$$
which can be made arbitrarily large for $\varepsilon$ arbitrarily small.
\par
\noindent
{\bf Conclusion}. The relevant $p$-range for $OS_p$ is $p \in [n, \infty]$.\\
We  note also that for $p \in [n, \infty]$, 
\begin{equation}\label{OS-ball}
OS_p(B^n_2)=n |B^n_2| = as_p(B^n_2).
\end{equation}
\vskip 2mm
\noindent
(iii) {\bf The case $os_p(K)$} \label{os}.
\par
\noindent
If $p=0$, then for all $K$, 
$$
os_0(K) =  \inf_{K'\in \mathcal K^K}\big(\as_0(K')\big) = n \inf_{K'\in \mathcal K^K} |K'| = n|K|.
$$
\par
\noindent
If $0 < p \leq \infty$ or if  $-\infty  <p  < -n$, then for all $K$, $os_p(K)=0$. 
Indeed,  for polytopes $P \in \mathcal K^K$,  we have for those $p$-ranges
$$
is_p(K) \leq \inf_{P\in \mathcal K^K} as_p(P) =0.
$$
\par
\noindent
{\bf Conclusion}. The relevant $p$-range for $os_p$ is $p \in (-n, 0]$.\\
We  note also that for $p \in (-n, 0]$, 
\begin{equation}\label{os-ball}
os_p(B^n_2)=n |B^n_2| = as_p(B^n_2).
\end{equation}
\vskip 2mm
\noindent
(iv) 
{\bf The case $is_p(K)$}.
\par
\noindent
We have that $is_p(K)=0$  for  all $p$ and for all $K$.
\newline
If  $0 < p \leq \infty$ or if $-\infty \leq p < -n$ we get  for polytopes $P \in \mathcal K_K$, 
$$
is_p(K) \leq \inf_{P\in \mathcal K_K} as_p(P) =0.
$$
If $-n<  p \leq 0$, then for all $K$, 
$$
is_p(K)  \leq \inf_{\varepsilon  B^n_2 \in \mathcal K_K}\big(\as_p(\varepsilon  B^n_2)\big) = 
n |B^n_2| \inf_{\varepsilon} \varepsilon ^{n \frac{n-p}{n+p}}=0.
$$
\par
\noindent
{\bf Conclusion}. There is no interesting $p$-range for the inner minimal affine surface area $is_p$.

\subsection{Continuity, monotonicity and isoperimetricity}

\vskip 2mm
\noindent
It was proved by Lutwak \cite{Lutwak96} that for $p\geq 1$, $L_p$-affine surface area is  an upper semicontinuous functional with respect to the Hausdorff metric.
In fact, it follows from Lutwak's proof that the same holds for all $0 \leq p <1$ (aside from the case $p=0$, which is just volume and hence continuous).
For $-n < p \leq 0$, the functional is lower semicontinuous. 
\vskip 2mm 
\noindent
\begin{prop}\label{lemma continuous}\cite{Lutwak96}
Let 
$0 \leq p \leq \infty$. Then the functional $K \mapsto \as_p(K)$ is upper semicontinuous with respect to the Hausdorff metric on $\R^n$.
For $- n < p \leq 0$, the functional is lower semicontinuous. 
\end{prop}
\vskip 2mm
\noindent
For the proof of the next lemma, we use Proposition \ref{lemma continuous} and the $L_p$-affine isoperimetric inequalities which were proved by Lutwak \cite{Lutwak96} for $p>1$ and for all other $p$ by 
Werner and Ye \cite{ WernerYe2008}. The case $p=1$ is the classical case.
\par
\noindent
For $p \geq 0$, 
\begin{align}\label{affine iso1}
\frac{\as_p(K)}{\as_p(B_2^n)} \le \left(\frac{|K|}{|B_2^n|}\right)^{\frac{n-p}{n+p}},
\end{align}
and for $-n < p \leq 0$, 
\begin{align}\label{affine iso2}
\frac{\as_p(K)}{\as_p(B_2^n)} \ge \left(\frac{|K|}{|B_2^n|}\right)^{\frac{n-p}{n+p}}.
\end{align}
Equality  holds in both inequalities iff $K$ is an ellipsoid.
Equality holds trivially in both inequalities if $p=0$.
\vskip 2mm
\noindent
\begin{lem}\label{lemma-extreme}  Let $K$ be  a convex body in $\mathbb{R}^n$. 
\par
\noindent
(i) Let $0\leq   p \leq n$.  Then there exists a convex body $K_0 \subset K$ such that
$$ IS_p(K) =   \sup_{C \in \mathcal K_K}\big(\as_p(C)\big)=as_p(K_0).$$
\vskip 2mm
\noindent
(ii) Let $n <  p \leq \infty$.  Then there exists a convex body $K_0 \supset K$ such that
$$ OS_p(K) =   \sup_{C \in \mathcal K^K}\big(\as_p(C)\big)=as_p(K_0).$$
\vskip 2mm
\noindent
(iii) Let $-n <  p < 0$.  Then there exists a convex body $K_0 \supset K$ such that
$$ os_p(K) =   \inf_{C \in \mathcal K^K}\big(\as_p(C)\big)=as_p(K_0).$$
\end{lem}
\begin{proof}
(i) When $p=0$, $IS_0(K) = n |K|$ and we take $K_0=K$ and when $p=n$,  $IS_n(K) = n |B^n_2|$ and we can take again $K_0=K$. Let now $0<p<n$.   By the $L_p$-affine isoperimetric inequality (\ref{affine iso1}),  $as_p(K)\leq  n \  |K|^{\frac{n-p}{n+p}} \  |B_2^n|^{\frac{2p}{n+p}}$,
and in particular, the supremum is finite. There is a sequence $(C_k)_{k \in \mathbb{N}}$ of convex bodies such that for all $k$, $C_k \subset K$ 
$$
as_p(C_k) + \frac{1}{k} \geq  \sup_{C \in \mathcal K_K}\big(\as_p(C)\big), 
$$
or
$$
\lim_{k \rightarrow \infty} as_p(C_k) =  \sup_{C \in \mathcal K_K}\big(\as_p(C)\big).
$$
By the Blaschke selection principle, see e.g., \cite{SchneiderBook}, there is a subsequence $(C_{k_i})_{i \in \mathbb{N}}$ that converges in 
Hausdorff distance to a convex set $K_0$. We claim that $K_0$ is a convex body in $\mathbb{R}^n$, i.e., $K_0$ has an interior point. Suppose not. Then 
$\lim_{i \rightarrow \infty} |C_{k_i}| =  |K_0| =0$. By the  $L_p$-affine isoperimetric inequality (\ref{affine iso1}),
$$
\lim_{i \rightarrow \infty} as_p(C_{k_i})  \leq \lim_{i \rightarrow \infty} n \  |C_{k_i}|^{\frac{n-p}{n+p}} \  |B_2^n|^{\frac{2p}{n+p}} =0.
$$
Therefore 
$$
0=\lim_{i \rightarrow \infty} as_p(C_{k_i}) = \sup_{C \in \mathcal K_K}\big(\as_p(C)\big) >0,
$$
which is a contradiction.
The last inequality holds since there is $\rho>0$ such that a ball with radius $\rho$ is contained in $K$. By the upper semi continuity of the $L_p$-affine surface area, 
$$
 \sup_{C \in \mathcal K_K}\big(\as_p(C)\big) = \limsup _{i \rightarrow \infty} as_p(C_{k_i}) \leq as_p(K_0)
 $$
 and thus $IS_p(K)=as_p(K_0)$.
 \vskip 2mm
 \noindent
 (ii) We can assume that $g(K) =0$. There is $R >0$ such that $K \subset RB^n_2$. 
 For all convex bodies $C$ such that  $C \supset K$ and $|C| \geq (n R)^n |B^n_2|$ there is a convex body $\tilde{C}$ such that
 $\tilde{C}\supset K$,  $\tilde{C} \subset n R B^n_2$ and $as_p(C) \leq as_p(\tilde{C})$. We now show the latter. There is an affine map $A$ 
 with determinant $1$ and $\rho >0$ such that $\rho B^n_2$ is the ellipsoid of maximal volume is contained in $A(C)$. Then by F. John's theorem 
 $$
 \rho B^n_2 \subset A(C) \subset n \rho B^n_2.
 $$
Therefore,
$
 (n R)^n |B^n_2| \leq |C| \leq  (n \rho)^n |B^n_2|$ and thus $R \leq \rho$. This yields
 $$
 K \subset R B^n_2 \subset \rho \  \frac{R}{\rho} \ B^n_2 \subset \frac{R}{\rho} A(C).
 $$
 We pick $\tilde{C} = \frac{R}{\rho} A(C)$. Then $as_p(\tilde{C}) = \left(\frac{R}{\rho}\right)^{n \frac{n-p}{n+p}}  as_p(C) \geq as_p(C)$,  as $\left(\frac{R}{\rho}\right)^{n \frac{n-p}{n+p}}\geq 1$,
 as $p>n$.
 If $C \supset K$ is such that $|C| \leq (n R)^n |B^n_2|$, we proceed as follows. As $K$ is a convex body, there is $r>0$ such that $rB^n_2 \subset K$ and thus $rB^n_2 \subset C$.
For every $x \in C$, let $x^\perp$ be the hyperplane through the origin and orthogonal to $x$. We consider the cone with base $x^\perp \cap rB^n_2$ and apex $x$. Let $h_x$ denote the height of the cone. Then we have for all $x\in C$ that 
$\frac{h_x}{n}  r^{n-1} |B^{n-1}_2| \leq  (n R)^n |B^n_2|$ and thus $C \subset \frac{ |B^n_2|}{|B^{n-1}_2| } \  \frac{n^{n+1} R^n}{r^{n-1}} \  B^n_2$.
 \par
 \noindent
Hence, altogether we can assume that the relevant (for the supremum) convex bodies $C \in \mathcal K^K$ are contained in $R_0 B^n_2$, where $R_0 = \max\left\{n R,  \frac{ |B^n_2|}{|B^{n-1}_2| } \  \frac{n^{n+1} R^n}{r^{n-1}}\right\}$.
We then proceed as  above. By the $L_p$-affine isoperimetric inequality (\ref{affine iso1}), we have for all relevant $C \in \mathcal K^K$ that  
$$
as_p(C)\leq  n \  |C|^{\frac{n-p}{n+p}} \  |B_2^n|^{\frac{2p}{n+p}} \leq n \  |K|^{\frac{n-p}{n+p}} \  |B_2^n|^{\frac{2p}{n+p}} ,
$$
and in particular the supremum is finite. Then, as above,  there is a sequence $(C_k)_{k \in \mathbb{N}}$ of convex bodies such that we have for all $k$ that $C_k \subset K$ and that
$$
as_p(C_k) + \frac{1}{k} \geq  \sup_{C \in \mathcal K^K}\big(\as_p(C)\big), 
$$
or
$$
\lim_{k \rightarrow \infty} as_p(C_k) =  \sup_{C \in \mathcal K^K}\big(\as_p(C)\big).
$$
By the Blaschke selection principle, see e.g., \cite{SchneiderBook}, there is a subsequence $(C_{k_i})_{i \in \mathbb{N}}$ that converges in 
Hausdorff distance to a convex set $K_0$. $K_0$ is a convex body as $ K \subset C_{k_i}$ for all $i$ and thus $K \subset K_0$. We conclude as in (i).
\vskip 2mm
 \noindent
 (iii) The proof is similar to (ii). We include it for completeness. We can again assume that $g(K) =0$ and that  there is $R >0$ such that $K \subset RB^n_2$. 
 As in (ii),  we claim that for all convex bodies $C$ such that  $C \supset K$ and $|C| \geq (n R)^n |B^n_2|$ there is a convex body $\tilde{C}$ such that
 $\tilde{C}\supset K$,  $\tilde{C} \subset n R B^n_2$ and $as_p(C) \geq as_p(\tilde{C})$. We now show this. There is an affine map $A$ 
 with determinant $1$ and $\rho >0$ such that $\rho B^n_2$ is the ellipsoid of maximal volume is contained in $A(C)$. Then by F. John's theorem 
 $$
 \rho B^n_2 \subset A(C) \subset n \rho B^n_2.
 $$
Therefore,
$
 (n R)^n |B^n_2| \leq |C| \leq  (n \rho)^n |B^n_2|$ and thus $R \leq \rho$. This yields
 $$
 K \subset R B^n_2 \subset \rho \  \frac{R}{\rho} \ B^n_2 \subset \frac{R}{\rho} A(C).
 $$
 We pick $\tilde{C} = \frac{R}{\rho} A(C)$. Then $as_p(\tilde{C}) = \left(\frac{R}{\rho}\right)^{n \frac{n-p}{n+p}}  as_p(C) \leq as_p(C)$,  as $\left(\frac{R}{\rho}\right)^{n \frac{n-p}{n+p}}\leq 1$,
 as $-n < p <0$.
 If $C \supset K$ is such that $|C| \leq (n R)^n |B^n_2|$, we proceed as follows. As $K$ is a convex body, there is $r>0$ such that $rB^n_2 \subset K$ and thus $rB^n_2 \subset C$.
For every $x \in C$, consider the cone with base $x^\perp \cap rB^n_2$ and apex $x$. Let $h_x$ denote the height of the cone. Then we have for all $x\in C$ that 
$\frac{h_x}{n}  r^{n-1} |B^{n-1}_2| \leq  (n R)^n |B^n_2|$ and thus $C \subset \frac{ |B^n_2|}{|B^{n-1}_2| } \  \frac{n^{n+1} R^n}{r^{n-1}} \  B^n_2$.
 \par
 \noindent
Hence, altogether we can assume that the relevant (for the infimum) convex bodies $C \in \mathcal K^K$ are contained in $R_0 B^n_2$, where $R_0 = \max\left\{n R,  \frac{ |B^n_2|}{|B^{n-1}_2| } \  \frac{n^{n+1} R^n}{r^{n-1}}\right\}$.
We then proceed as  above. By the $L_p$-affine isoperimetric inequality (\ref{affine iso2}), we have for all relevant $C \in \mathcal K^K$ that  
$$
as_p(C)\geq  n \  |C|^{\frac{n-p}{n+p}} \  |B_2^n|^{\frac{2p}{n+p}} \geq n \  |K|^{\frac{n-p}{n+p}} \  |B_2^n|^{\frac{2p}{n+p}} ,
$$
and in particular the infimum is finite. Then, as above,  there is a sequence $(C_k)_{k \in \mathbb{N}}$ of convex bodies such that  $C_k \subset K$ for all $k$ and such that 
$$
as_p(C_k)  \leq  \inf_{C \in \mathcal K^K}\big(\as_p(C)\big) + \frac{1}{k} , 
$$
for all $k$, and hence
$$
\lim_{k \rightarrow \infty} as_p(C_k) =  \inf_{C \in \mathcal K^K}\big(\as_p(C)\big).
$$
By the Blaschke selection principle, see e.g., \cite{SchneiderBook}, there is a subsequence $(C_{k_i})_{i \in \mathbb{N}}$ that converges in 
Hausdorff distance to a convex set $K_0$. $K_0$ is a convex body as $ K \subset C_{k_i}$ for all $i$ and thus $K \subset K_0$. Again, we conclude as in (i).
\end{proof}
\vskip 2mm
\noindent
It is natural to ask about the continuity properties of inner and outer maximal, respectively minimal,  affine surface areas in the  $p$-ranges that are not
already settled by the above considerations.
\par
\noindent
\vskip 2mm
\noindent
\begin{prop}\label{prop continuous}
Let the set of convex bodies in $\R^n$ be endowed with the  Hausdorff metric.
\par
\noindent
For $0 \leq p \leq n$,  the functional $K \mapsto IS_p(K)$ is continuous.
\par
\noindent
For $n\leq p \leq \infty$,  the functional $K \mapsto OS_p(K)$ is continuous.
\par
\noindent
For $-n < p \leq 0$,  the functional $K \mapsto os_p(K)$ is continuous.
\end{prop}
\vskip 2mm
\noindent
The next proposition lists  affine isoperimetric inequalities and  monotonicity properties for the  the functionals $IS_p$, $OS_p$ and $os_p$.
\par
\noindent
\begin{prop}\label{monotonicity}  Let $K$ be  a convex body in $\mathbb{R}^n$. 
\par
\noindent
(i) Let $0\leq   p \leq n$.  Then
$ IS_p(K)\leq IS_p(B^n_2) \  \left(\frac{|K|}{|B^n_2|} \right) ^\frac{n-p}{n+p}.$
\newline
Equality holds trivially if  $p=0$ or  $p=n$. 
\par
\noindent
Let $n \leq   p \leq \infty$.  Then
$ OS_p(K)\leq OS_p(B^n_2) \  \left(\frac{|K|}{|B^n_2|} \right) ^\frac{n-p}{n+p}.$
\newline
Equality holds trivially if   $p=n$. 
\par
\noindent
Let  $-n< p \leq 0$. Then 
$os_p(K)\geq os_p(B^n_2) \  \left(\frac{|K|}{|B^n_2|} \right) ^\frac{n-p}{n+p}.$
\newline
Equality holds trivially if   $p=0$. 
\par
\noindent
For all other $p$, equality holds in all  inequalities  iff $K$ is an ellipsoid. 
\vskip 2mm
\noindent
(ii)  $p \rightarrow \left(\frac{IS_p(K)}{ n |K|}\right)^\frac{n+p}{p}$ is strictly increasing in $p \in (0, n]$. 
\par
\noindent
$ p \rightarrow \left(\frac{OS_p(K)}{ n |K^\circ|}\right)^\frac{n+p}{p}$ is strictly decreasing in $p \in [n,  \infty)$.
\par
\noindent
$ p \rightarrow \left(\frac{os_p(K)}{ n |K^\circ|}\right)^\frac{n+p}{p}$ is strictly decreasing in $p \in (-n, 0)$.
\end{prop}
\vskip 3mm
\noindent
\subsection{Asymptotic estimates}

The next theorems provide estimates for the inner and outer extremal  affine surface areas in the  $p$-ranges that are not
already settled above. There,  $L_K$  is the  isotropic constant of $K$ as defined in \eqref{LK}.
\begin{thm}\label{thm asymp}
There is an absolute  constant $C>0$ such that for all $n \in \mathbb{N}$, all $0 \leq p \leq n$ and all convex bodies 
$K\subseteq \R^n$,
\begin{align}\label{asymp upper}
\frac{1}{n^{5/6}}  \  \left(\frac{C }{L_K}\right)^{\frac{2np}{n+p}}  \  \frac{IS_p(B^n_2)}{ |B^n_2|^\frac{n-p}{n+p}}\leq \  \frac{IS_p(K)}{ |K|^\frac{n-p}{n+p}} \leq  \   \frac{IS_p(B^n_2)}{ |B^n_2|^\frac{n-p}{n+p}}.
\end{align}
Equality holds trivially in the right inequality if $p=0, n$. 
If $p \neq 0, n$, equality holds in the right inequality iff $K$ is a centered ellipsoid.
\end{thm}
\vskip 2mm
\noindent
By (\ref{IS-ball}), $ \frac{IS_p(B^n_2)}{ |B^n_2|^\frac{n-p}{n+p}} = n |B^n_2|^ \frac{2p}{n+p}$. Therefore, Theorem \ref{thm asymp} states that
\begin{align*}
\frac{1}{n^{5/6}}  \  \left(\frac{C }{L_K}\right)^{\frac{2np}{n+p}}  \  \leq \  \frac{IS_p(K)}{n \  |B^n_2|^\frac{2p}{n+p} \   |K|^\frac{n-p}{n+p}}\   \leq  \   1.
\end{align*}
Stirling's formula yields that with absolute constants, $c_1, c_2$, 
$$
 \frac{c_2^\frac{np}{n+p}}{n^\frac{n(p-1)-p}{n+p}} \leq \frac{IS_p(B^n_2)}{ |B^n_2|^\frac{n-p}{n+p}} = n |B^n_2|^ \frac{2p}{n+p} \leq  \frac{c_1^\frac{np}{n+p}}{n^\frac{n(p-1)-p}{n+p}}.
$$
\par
\noindent
As noted, the upper bound  is sharp when e.g.,  $K$ is $B_2^n$.
However, in general we have $IS_p(K) > \as_p(K)$. For example, for the $n$-dimensional cube $B^n_\infty$ centered at $0$ with sidelength $2$, 
$\as_p(B_\infty^n) = 0$, but $B_2^n \subseteq B_\infty^n$ and so $IS_p(B_\infty^n) \ge \as(B_2^n) > 0$.
\vskip 3mm
\begin{thm}\label{thm asymp2}
Let $K\subseteq \R^n$ be a convex body. 
\par
\noindent
(i) 
Let $n \leq p \leq \infty$.  Then there are absolute constants $c$  and $C$ such that 
\begin{align}\label{asymp upper}
\max \left\{ n^{-5/6}c^{n\frac{p-n}{p+n} }\left( \frac{C}{L_{(K-s(K))^\circ}}\right)^{\frac{2n^2}{n+p}}, n^{n \frac{n-p}{n+p}} \right\} \  \frac{OS_p(B^n_2)}{ |B^n_2|^\frac{n-p}{n+p}} \   \leq \  \frac{OS_p(K)}{ |K|^\frac{n-p}{n+p}} \leq  \   \frac{OS_p(B^n_2)}{ |B^n_2|^\frac{n-p}{n+p}},
\end{align}
where $s(K)$ is the Santal{\'o} point of $K$. Equality holds trivially in the right inequality if $p=n$. 
If $p \neq  n$, equality holds in the right inequality iff $K$ is a centered ellipsoid.
\par
\noindent
(i) 
Let $-n < p \leq0$.  Then 
\begin{align}\label{asymp upper}
 \frac{os_p(B^n_2)}{ |B^n_2|^\frac{n-p}{n+p}}\leq \  \frac{os_p(K)}{ |K|^\frac{n-p}{n+p}} \leq  \   n^{n \frac{n-p}{n+p}}\  \frac{os_p(B^n_2)}{ |B^n_2|^\frac{n-p}{n+p}}.
\end{align}
Equality holds trivially in the left inequality if $p=0$. 
If $p \neq  0$, equality holds in the left inequality iff $K$ is a centered ellipsoid.
\end{thm}
\par
\noindent
If $p=n$, then the maximum in the lower bound of (i) is achieved for the second term and is $1$. If $p= \infty$, the maximum is achieved for the 
first term and it is equal to $n^{-5/6}c^{n}$. 
\newline
If $K$ is centrally symmetric, $n^{n \frac{n-p}{n+p}}$  can be replaced by $n^{n \frac{n-p}{2(n+p)}}$.

\subsection{Relation to quermassintegrals}
\vskip 2mm
\noindent
Finally, we turn to the relation of  the extremal affine surface areas  to quermassintegrals. While some of the (trivial) extremal affine surface areas are quermassintegrals,
we will see that in general this is not the case.
\par
\noindent
Given a convex body $K \subseteq \R^n$ and $t \ge 0$, the Steiner formula (see,  for example \cite{SchneiderBook}) says that there exist non-negative numbers $W_0(K),\dots, W_n(K)$, such that
\begin{align*}
|K+t \, B_2^n| = W_0(K)  + \binom{n}{1} W_1(K) t + \binom{n}{2}W_2(K) t^2 + \dots +W_n(K) t^n.
\end{align*}
The numbers $W_0(K), \dots, W_n(K)$ are called the quermassintegrals. In particular,  $W_0(K)=  |K|$ and $W_n(K) =  |B^n_2|$. 
Therefore,  by section \ref{Relevant-p}, $IS_0(K) = os_0(K) = n |K| = n W_0(K)$ and $IS_n(K) = OS_n(K) = n |B^n_2| = n W_n(K)$ are (multiples of) quermassintegrals. 
However, as shown in the next proposition, in general the extremal affine surface areas are  not (multiples, or powers of)  quermassintegrals. 
\par
\noindent
We only treat the cases $IS_1$, $os_{-1}$ and $OS_{n^2}$. The other  relevant $p$-cases  are treated similarly.
\vskip 2mm
\noindent
\begin{prop}\label{thm quer}
(i)  If $\beta >0$,  then  $IS_1^\beta$ and $os_{-1}$ are not equal to $W_i$, for any $0 \le i \le n$,  and if $\beta <0$,  then  $OS_{n^2} ^\beta$ is not equal to $W_i$, for any $0 \le i \le n$.
\vskip 2mm
\noindent
(ii) 
The quantities $IS_1$, $os_{-1}$ and $OS_{n^2}$ are not a linear combination of quermassintegrals.  In particular, those quantities are  not  valuations.
\end{prop}
\smallskip
\noindent
\begin{rmk}
From~\cite{Schuett1993} it is known that affine surface area is a valuation, that is, for every $K,L \subseteq \R^n$ convex, 
\begin{align*}
\as_1(K\cap L) + \as_1(K\cup L) = \as_1(K) + \as_1(L).
\end{align*}
It is also known by Hadwiger's characterization theorem~\cite{HadwigerBook}, that every continuous rigid motion invariant valuation on the set of convex bodies is a linear combination of quermassintegrals. 
Thus, Proposition~\ref{thm quer} (ii) shows in particular that $IS_1$, $os_{-1}$ and $OS_{n^2}$ are  not valuations.
\end{rmk}
\medskip
\section{Proofs}
\vskip 2mm
\noindent
\begin{proof}[Proof of Proposition~\ref{prop continuous}]
By section (\ref{Relevant-p}) (i), $IS_0(K) =n|K|$ is just volume, which is continuous and $IS_0(K) =n|B^n_2|$,  which is constant and hence continuous.
Thus for $IS_p(K)$ we only need to consider $p \in (0,n)$. 
We may assume  that $0$ is the center of gravity of $K$, that is,
\begin{align*}
\int_{K}x\, dx = 0.
\end{align*}
Hence, there exists $\rho >0$ such that $\rho B^n_2 \subseteq K$. 
Let $\{K_l\}_{l=1}^{\infty}$ be a sequence of convex bodies, all having center of gravity at the origin,  that converges to $K$ in the Hausdorff metric. That is, for 
every $\varepsilon >0$, there exists $l_0\in \mathbb N$ such that for all $l \geq l_0$, 
$$
K_l \subseteq K + \varepsilon B^n_2 \  \   \text{and}  \  \   K \subseteq K_l + \varepsilon B^n_2.
$$
If $\varepsilon>0$ is sufficiently small, then we can assume that for all $l \geq l_0$, $\frac{\rho}{10} B^n_2 \subseteq K_l$.
Thus, for all $l \geq l_0$,
\begin{equation}\label{contain1}
K_l \subseteq K + \varepsilon B^n_2  \subseteq K + \frac{\varepsilon}{\rho} K = \left( 1 +\frac{\varepsilon}{\rho} \right) K,
\end {equation}
and 
\begin{equation}\label{contain2}
K \subseteq K_l + \varepsilon B^n_2  \subseteq K_l + \frac{10 \varepsilon}{\rho} K_l = \left( 1 +\frac{10 \varepsilon}{\rho} \right) K_l.
\end {equation}
Hence,
\begin{eqnarray*}
\left(1+ \frac{\varepsilon}{\rho}\right)^{n\frac{n-p}{n+p}} IS_p(K) \  {=} \  
IS_p\left( \left(1+ \frac{\varepsilon}{\rho}\right) K\right) \stackrel{\eqref{contain1}}{\ge} IS_p(K_l), 
\end{eqnarray*}
and
\begin{eqnarray*}
\left(1+ \frac{10\varepsilon}{\rho}\right)^{n\frac{n-p}{n+p}} IS_p(K_l)  \  {=}  \  
IS_p\left( \left(1+ \frac{10\varepsilon}{\rho}\right) K_l\right) \stackrel{\eqref{contain2}}{\ge} IS_p(K).
\end{eqnarray*}
In the last two lines above, we have also used (\ref{pafteraffine}), resp. the remark after it.
Altogether, for all $l \ge l_0$,
\begin{align*}
\left(1+\frac\varepsilon \rho\right)^{-n\frac{n-p}{n+p}} IS_p(K_l) \le IS_p(K) \le \left(1+\frac{10\varepsilon}{\rho}\right)^{n\frac{n-p}{n+p}} IS_p(K_l).
\end{align*}
Since $\varepsilon>0$ is arbitrary, the result follows.
\par
\noindent
Continuity for outer maximal affine surface area $OS_p$ and outer minimal  surface area $os_p$ is treated similarly.
\end{proof}
\vskip 2mm
\noindent
For the proof of Proposition~\ref{monotonicity}  and Theorem ~\ref{thm asymp}, we use the above quoted $L_p$-affine isoperimetric inequalities.
\vskip 2mm
\noindent
\begin{proof}[Proof of Proposition~\ref{monotonicity}]
\par
\noindent
(i) When $0 < p \leq n$ and  $K' \subseteq K$, we use (\ref{affine iso1}) and  (\ref{IS-ball}),
\begin{eqnarray*} 
IS_p(K) &=&  \sup_{K'\in \mathcal K_K}\big(\as_p(K')\big) \leq  \sup_{K'\in \mathcal K_K}  \as_p(B_2^n)\left(\frac{|K'|}{|B_2^n|}\right)^{\frac{n-p}{n+p}}
\leq \as_p(B_2^n) \ \left(\frac{|K|}{|B_2^n|}\right)^{\frac{n-p}{n+p}} \\
&= &IS_p(B^n_2)  \ \left(\frac{|K|}{|B_2^n|}\right)^{\frac{n-p}{n+p}} .
\end{eqnarray*}
From the equality characterization of  (\ref{affine iso1}) it follows that equality holds iff $K$ is an ellipsoid.
\newline
Similarly, we get  for $OS_p$ when $p \in (n, \infty]$, also using (\ref{OS-ball}),
\begin{eqnarray*} 
OS_p(K) &=&  \sup_{K'\in \mathcal K^K}\big(\as_p(K')\big) \leq  \sup_{K'\in \mathcal K^K}  \as_p(B_2^n)\left(\frac{|K'|}{|B_2^n|}\right)^{\frac{n-p}{n+p}}
\leq \as_p(B_2^n) \ \left(\frac{|K|}{|B_2^n|}\right)^{\frac{n-p}{n+p}} \\
&= &OS_p(B^n_2)  \ \left(\frac{|K|}{|B_2^n|}\right)^{\frac{n-p}{n+p}}.
\end{eqnarray*}
From the equality characterization of  (\ref{affine iso1}) it follows that equality holds iff $K$ is an ellipsoid.
\newline
In the same way, using (\ref{affine iso2})  and (\ref{os-ball}) when $-n< p < 0$, we have 
\begin{eqnarray*} 
os_p(K) &=&  \inf_{K'\in \mathcal K^K}\big(\as_p(K')\big) \geq   \as_p(B_2^n)  \inf_{K'\in \mathcal K^K} \left(\frac{|K'|}{|B_2^n|}\right)^{\frac{n-p}{n+p}} \\
&\geq &
os_p(B^n_2)  \ \left(\frac{|K|}{|B_2^n|}\right)^{\frac{n-p}{n+p}}.
\end{eqnarray*}
Equality characterization follows from the equality characterization of (\ref{affine iso2}).
\vskip 2mm
\noindent
(ii)  It was shown in \cite{WernerYe2008} (see also \cite{PaourisWerner2012}) that   
the function $p \rightarrow \left(\frac{as_{p}(K)}{n |K|}\right)^{\frac{n+p}{p}}$ is  strictly increasing in $p\in (0, \infty)$.  
Therefore we get for $0 <p < q \leq n$, 
\begin{eqnarray*}
\left(\frac{IS_p(K)}{n |K|}\right)^{\frac{n+p}{p}} &=& \frac{\sup_{K'\in \mathcal K_K}\big(\as_p(K')^{\frac{n+p}{p}} \big) }{(n |K|)^{\frac{n+p}{p}} } < \frac{\sup_{K'\in \mathcal K_K}
\left(n |K'|\right) ^{\frac{n}{p}- \frac{n}{q}} \  \big(\as_q(K')^{\frac{n+q}{q}} \big) } {(n |K|)^{\frac{n+p}{p}} } \\
&\leq& \frac{\left(n |K|\right)^{\frac{n}{p}-\frac{n}{q}}}{(n |K|)^{\frac{n+p}{p}}} \  \left(\sup_{K'\in \mathcal K_K} \as_q(K') \right) ^{\frac{n+q}{q}}   = \left(\frac{IS_q(K)}{n |K|}\right)^{\frac{n+q}{q}}.
\end{eqnarray*}
It was also shown in \cite{WernerYe2008} (see also \cite{PaourisWerner2012}) that   
the function $p \rightarrow \left(\frac{as_{p}(K)}{n |K^\circ |}\right)^{n+p}$ is  strictly decreasing in $p\in (0, \infty)$. Therefore we get 
for $n \leq p < q < \infty$, 
\begin{eqnarray*}
\left(\frac{OS_p(K)}{n |K^\circ|}\right)^{n+p} &=& \frac{\sup_{K'\in \mathcal K^K}\big(\as_p(K')^{n+p} \big)}{(n |K^\circ|)^{n+p}} 
> \frac{\sup_{K'\in \mathcal K^K}
\left(n |K^{'\circ} | \right)^{p-q} \  \big(\as_q(K')^{n+q} \big)} {(n |K^\circ|)^{n+p} } \\
&\geq& \frac{ \left(\sup_{K'\in \mathcal K^K} \as_q(K') \right)^{n+q} }{\left(n |K^\circ|\right)^{n+q}}  = \left(\frac{OS_q(K)}{n |K^\circ|}\right)^{n+q}
\end{eqnarray*}
and
for $-n \leq p < q \leq 0$, 
\begin{eqnarray*}
\left(\frac{os_p(K)}{n |K^\circ|}\right)^{n+p} &=& \frac{\inf_{K'\in \mathcal K^K}\big(\as_p(K')^{n+p} \big)}{(n |K^\circ|)^{n+p}} 
> \frac{\inf_{K'\in \mathcal K^K}
\left(n |K^{'\circ} | \right)^{p-q} \  \big(\as_q(K')^{n+q} \big)} {(n |K^\circ|)^{n+p} } \\
&\geq& \frac{ \left(\inf_{K'\in \mathcal K^K} \as_q(K') \right)^{n+q} }{\left(n |K^\circ|\right)^{n+q}}  = \left(\frac{OS_q(K)}{n |K^\circ|}\right)^{n+q}
\end{eqnarray*}
\end{proof}
\vskip 2mm
\noindent
In part of the proof below it  is most convenient to work with a body which is in isotropic position. A body $K\subseteq \R^n$ is said to be in isotropic position if $|K| =1 $ and there exists $L_K > 0$ such that for all $\theta \in \S^{n-1}$,
\begin{align*}
\int_K\langle x,\theta\rangle dx = 0, \qquad \int_K \langle x,\theta \rangle ^2dx = L_K^2.
\end{align*}
Here and in what follows, $\S^{n-1}$ denotes the unit Euclidean sphere in $\R^n$. It is known that for every convex body $K\subseteq \R^n$, there exists $T:\R^n \to \R^n$ affine and invertible such that $TK$ is isotropic.  See for example~\cite{BGVV14} for this and other facts on isotropic position used here.
\vskip 2mm
\noindent
\begin{proof}[Proof of Theorem~\ref{thm asymp}] 
\vskip 2mm
\noindent
The upper bound, together with the equality characterizations,  follows immediately from Proposition \ref{monotonicity} (i).
\par
\noindent
Now we turn to  the lower bound in the case (i). 
As noted above, $IS_p(TK)=\det(T)^{\frac{n-p}{n+p}}IS_p(K)$ for any invertible linear map $T$. Therefore, 
to prove the lower bound  for $0 < p < \infty$,  it is sufficient to consider $K$ in isotropic position. 
Let 
$L_K$ be the isotropic constant of $K$. By the thin shell estimate of O. Gu\'edon and E.Milman \cite{GuedonMilman} (see also \cite{Fleury, LeeVempala17, Paouris}), we have with  universal constants $c$ and $C$, that for all $t \geq 0$, 
\begin{eqnarray*} 
\left| K \cap \left\{ x\in \mathbb{R}^n :\, \left| \|x\|-L_K\sqrt{n} \right| <  t L_K\sqrt{n} \right\}\right| >1-  C \text{exp}(-cn^{1/2}\text{min}(t^3,\, t)).
\end{eqnarray*}
Taking $t= O(n^{-1/6})$, there is a  a new universal constant $c>0$ such that for all $n \in \mathbb{N}$, 
\begin{equation}
\left| K \cap \left\{ x\in \mathbb{R}^n: \, \left| \|x\|-L_K\sqrt{n} \right| < c L_K n^{1/3} \right\}\right| \geq \frac{1}{2}.  \label{thin shell}
\end{equation}
This  set consists of all  $x \in K$  for which
 $$
 L_K \left( n^{1/2} - c n^{1/3}\right) <  \|x\|  <  L_K \left( n^{1/2} +  c n^{1/3}\right).
 $$ 
We consider those $n \in \mathbb{N}$ for which 
$n^{1/6} >c$.
\newline
We will truncate the above set. For $i=0,\,1,\,2,\,\dots k_n=  \lfloor{n \log_2 \frac {n^{1/2}+ c n^{1/3}}{n^{1/2}- c n^{1/3}}} \rfloor$, consider the sets 
$$
    L_i := K\cap \{x\in \mathbb{R}^n \, :\, 2^{i/n} (L_K (n^{1/2}-c n^{1/3}) )< \|x\| \le 2^{(i+1)/n}(L_K(n^{1/2}-cn^{1/3}))\}.
$$
Then
$$
2^{\frac{k_n}{n} } \leq 2^{\log_2 \frac {n^{1/2}+ c n^{1/3}}{n^{1/2}- c n^{1/3}}} = \frac {n^{1/2}+ c n^{1/3}}{n^{1/2}- c n^{1/3}}
$$
and thus
\begin{align}
    K \cap \left\{ x\in \mathbb{R}^n\,  :\left| \|x\|-L_K\sqrt{n} \right| < c L_K n^{1/3} \right\} \subset \cup_{i=0}^{k_n} L_i \label{truncation}.
\end{align}
Moreover, with a new absolute constant $C_0$, 
$$
k_n \leq n \log_2 \frac {n^{1/2}+ c n^{1/3}}{n^{1/2}- c n^{1/3}} = n  \log_2 \frac {1 + c n^{-1/6}}{ 1 - c n^{-1/6}} \leq  \ C_0\  n^{5/6}.
$$
By (\ref{thin shell}) and (\ref{truncation}), there exists $i_0 \in \{1,\,2,\,\dots,\, \lfloor {C_0 n^{5/6}}\rfloor \}$ such that 
\begin{equation}
|L_{i_0} |\ge \frac{1}{2 \   \lfloor {C_0 n^{5/6}}\rfloor}. \label{eq:Livolume}
\end{equation}
We set $R=2^{i_0/n}(L_K( n^{1/2}-c n^{1/3}))$. In particular, we have 
$$L_{i_0}=K\cap \{x\in \mathbb{R}^n \, :\, R < \|x\| \le 2^{1/n}R\}.$$ Let 
\[
O=\left\{ \theta\in S^{n-1}\,:\, \rho_{K}\left(\theta\right)>R \right\} ,\,\text{and}\,\,S_{O}=\left\{ r\theta\,:\,\theta\in O\,{\rm and}\,r\in\left[0,\,R\right]\right\} \subset K, 
\]
where $\rho_{K}\left(\theta\right)=\max\left\{ r\ge0\,:\,r\theta\in K\right\} $
is the radial function of $K$.
\par
\noindent
Now we claim that 
\begin{equation}
L_{i_0} \subset 2^{1/n}S_O.
\label{eq:inclusion}
\end{equation}
Indeed, let  $y\in L_{i_0}$. We express $y=r\theta$ in polar coordinates. By definition, we have $R<r<2^{1/n}R$ and $r\theta \in K$. Thus, $\rho_K(\theta)\ge r>R$ and hence $\theta\in O$. Therefore, $r\theta\in 2^{1/n}S_O$ because $r\in [0,\,2^{1/n}R]$.
By (\ref{eq:Livolume}) and (\ref{eq:inclusion}) we conclude that
\begin{equation}
\left|S_{O}\right|\ge\left(2^{-1/n}\right)^{n}|L_{i_0}| \ge \frac{1}{4 \   \lfloor {C_0 n^{5/6}}\rfloor}.\label{eq:volume}
\end{equation}
Now, we consider $as_{p}\left(K\cap RB_{2}^{n}\right)$.
For $\theta\in O$, $R \theta$ is a boundary
point of $K\cap R B_{2}^{n}$. Thus,
\begin{align*}
as_{p}\left(K\cap R B_{2}^{n}\right) & \ge\int_{RO}\frac{\kappa^{\frac{p}{n+p}}}{\iprod x{N\left(x\right)}^{\frac{n\left(p-1\right)}{n+p}}}d{\mu\left(x\right)}
  =\int_{RO}\frac{R^{-\left(n-1\right)\frac{p}{n+p}}}{R^{\frac{n\left(p-1\right)}{n+p}}}d{\mu\left(x\right)}\\
 & =\mu\left(RO\right)\left(\frac{1}{R}\right)^{\frac{\left(n-1\right)p+n\left(p-1\right)}{n+p}}
   =\mu\left(RO\right)\left(\frac{1}{R}\right)^{\frac{2np}{n+p}-1},
\end{align*}
where $\mu$ is the surface area measure of $RS^{n-1}$.
We can compare surface area and volume, 
\[
\frac{\mu\left(RO\right)\cdot R}{n}=\left|S_{O}\right|.
\]
Hence,
\begin{align*}
as_{p}\left(K\cap RB_{2}^{n}\right) & \ge\left(\frac{1}{R}\right)^{\frac{2np}{n+p}-1}\frac{n}{R}\left|S_{O}\right|=\left(\frac{1}{R}\right)^{\frac{2np}{n+p}}n\left|S_{O}\right|\\
 & \ge\left(\frac{1}{R}\right)^{\frac{2np}{n+p}} \frac{n}{4 \   \lfloor {C_0 n^{5/6}}\rfloor}.
\end{align*}
Since $R\le 2\sqrt{n}L_K$, this finishes the proof for the lower bound. 
\end{proof}

\vskip 3mm
\noindent
\begin{proof}[Proof of Theorem~\ref{thm asymp2}] 
The upper bound of (i) and the lower bound of (ii), together with the equality characterizations,  follow immediately from Proposition \ref{monotonicity} (i).
\par
\noindent
For the other estimates, we will rely on the dual body of $K$. Recall that the Santal{\'o} point $s(K)$ of a convex body $K$ is the unique point $s(K)$ for which 
 the origin is the barycenter  of $(K-s(K))^\circ$. Without loss of generality, we may assume $s(K)=0$
and $K^\circ$ is in isotropic position. 
\par
\noindent
Following the proof of Theorem~\ref{thm asymp}, there exists $R \in [\frac{1}{2} \sqrt{n}L_{K^\circ}, 2\sqrt{n}L_{K^\circ}]$ such that 
$$
O: = \{ \theta \in S^{n-1} \,:\, \rho_{K^\circ}(\theta)>R\}
$$
satisfies 
\[
\mu\left(RO\right) \ge \frac{n}{R} \frac{1}{Cn^{5/6}},
\]
where $\mu$ is the surface area measure of $RS^{n-1}$.
\par
\noindent
We consider the following convex hull $ {\rm conv}\{R^2K,\, RB_2^n\}$. 
Recall that the support function of a convex body $L$ is $h_L(\theta):= \max_{x\in L} \langle x, \theta \rangle$. Furthermore, we have the identity $h_L(\theta) = \frac{1}{\rho_{L^\circ}(\theta)}$. Thus, for $\theta \in O$
\[
	h_{R^2K}(\theta)=R^2h_K(\theta) = R^2 \frac{1}{\rho_{K^\circ}(\theta)} < R.  
\]
Therefore, $R\theta \in \partial ( {\rm conv}\{R^2K,\, RB_2^n\})$. 
Now, we have 
\[
OS_p(R^2K) \geq  as_p( {\rm conv}\{R^2K,\, RB_2^n\}) 
\geq \int_{RO}\frac{R^{-\left(n-1\right)\frac{p}{n+p}}}{R^{\frac{n\left(p-1\right)}{n+p}}}d{\mu\left(x\right)}
= \frac{n}{Cn^{5/6}}\left(\frac{1}{R}\right)^{\frac{2np}{n+p}}.
\]
Using the fact that $|K^\circ|=1$ and the volume product estimate $|L||L^\circ| \ge c^n    |B_2^n|^2$ from Theorem 1 of \cite{BM87}, we have
\[
	|R^2K|\ge c^n  R^{2n}|B_2^n|^2
\]
for some constant $c>0$.
\par
\noindent
Alltogether we conclude 
\[
	\frac{OS_p(R^2K)}{|R^2K|^{\frac{n-p}{n+p}}} \ge\frac{n}{Cn^{5/6}}\left(\frac{1}{R}\right)^{\frac{2np}{n+p}} (c^n R^{2n}|B_2^n|^2)^{\frac{p-n}{n+p}}
	= \frac{n}{Cn^{5/6}}\left(\frac{1}{R}\right)^{\frac{2n^2}{n+p}} (c^n |B_2^n|^2)^{\frac{p-n}{n+p}}.
\]
With the identity $
\frac{OS_p(B^n_2)}{ |B^n_2|^\frac{n-p}{n+p}} = n |B^n_2|^ \frac{2p}{n+p}$, we obtain 
\[
	\frac{OS_p(R^2K)}{|R^2K|^{\frac{n-p}{n+p}}} 
	\ge \frac{OS_p(B^n_2)}{ |B^n_2|^\frac{n-p}{n+p}}  \frac{1}{n^{5/6}}c^{n\frac{p-n}{p+n}} \left(\frac{C}{L_{K^\circ}}\right)^{\frac{2n^2}{n+p}}.
\]
\vskip 2mm
\noindent
Furthermore, we can derive a different bound using L\"owner position. We will assume that $K$ is in L\"owner position, i.e.,  the L\"owner ellipsoid $L(K)$, which is the ellipsoid of minimal volume containing $K$,
is the Euclidean ball $\frac{|L(K)|}{|B^n_2|} B^n_2$.  We also have that 
\begin{equation}\label{loewner}
K \subset L(K) \subset n \ K,
\end{equation}
and that for a $0$-symmetric convex body $K$,
\begin{equation}\label{loewner2}
K \subset L(K) \subset \sqrt{n}  \ K.
\end{equation}
\par
\noindent
(i) We get with  (\ref{pafteraffine}),  (\ref{OS-ball}) and (\ref{loewner}),
\begin{eqnarray*}
OS_p(K) \geq as_p(L(K)) = \left(\frac {|L(K)|}{|B^n_2|} \right) ^\frac{n-p}{n+p}  n \  |B^n_2|  \geq n^{n \frac {n-p}{n+p}} \left(\frac {|K|}{|B^n_2|} \right) ^\frac{n-p}{n+p} \  OS_p(B^n_2),
\end{eqnarray*}
which finishes the lower estimate of (i).
\vskip 2mm
\noindent
(ii) Similarly, we get with (\ref{pafteraffine}),  (\ref{OS-ball}) and (\ref{loewner}),
\begin{eqnarray*}
os_p(K) \leq as_p(L(K)) = \left(\frac {|L(K)|}{|B^n_2|} \right) ^\frac{n-p}{n+p}  n \  |B^n_2|  \leq n^{n \frac {n-p}{n+p}} \left(\frac {|K|}{|B^n_2|} \right) ^\frac{n-p}{n+p} \  os_p(B^n_2).
\end{eqnarray*}
In the $0$ -symmetric case we use (\ref{loewner2}) to get the estimate with $n^{n \frac {n-p}{2(n+p)}}$ instead of $n^{n \frac {n-p}{n+p}}$.
\end{proof}

\vskip 3mm
\noindent
\begin{proof}[Proof of Proposition~\ref{thm quer}]\
We only give the proofs for $IS_1$.  The proofs for $os_{-1}$ and $OS_{n^2}$ are the same with the obvious modifications.
\par
\noindent
(i) To prove the first assertion, note that by (the remark after) (\ref{pafteraffine}), $IS_1^\beta$ is homogeneous of degree $\frac{\beta n(n-1)}{n+1}$. Also, it is known that $W_i$ is homogeneous of degree $n-i$. Hence, if $IS_1^\beta = W_i$ for some $i$, then $\frac{\beta n(n-1)}{n+1} \in \mathbb N$ and in particular $\beta \in \mathbb Q$. On the other hand, it is known that $W_i(B_2) = |B_2^n|$. Thus, we must also have
\begin{align*}
 |B_2^n| = IS_p^\beta(B_2^n) \stackrel{(*)}{=}  n^{\beta}|B_2^n|^\beta,
\end{align*}
where in ($*$) we used  (\ref{IS-ball}). Therefore, we have $|B_2^n|^{\frac{1-\beta}{\beta}} \in \mathbb N$. Now, it is known that
\begin{align*}
|B_2^n| = \frac{\pi^\frac n 2}{\Gamma\left(\frac n 2 + 1\right)} = \begin{cases} \frac{\pi^\frac n 2}{(n/2)!} & 2 \mid n, \\ \frac{2((n-1)/2)!(4\pi)^\frac {n-1} 2}{n!} & 2 \nmid n. \end{cases}
\end{align*}
In other words, for every $n \in \mathbb N$, we have $|B_2^n| = Q_n \pi^\frac n 2$ or $|B_2^n| = Q_n \pi^\frac {n-1} 2$, where $Q_n \in \mathbb Q$. Therefore, if $|B_2^n|^{\frac{1-\beta}{\beta}} \in \mathbb N$ with $\beta \in \mathbb Q$, that would imply that $\pi$ is an algebraic number, which is not the case. This proves the first assertion.
\vskip 2mm
\noindent
(ii) Suppose that $IS_1$ is a linear combination of quermassintegrals. Then,  for $K$ given, there exist $\lambda_i$, $0 \leq i \leq n$, not all of them equal to $0$, such that
$IS_1(K) = \sum_{i=0}^n \lambda_i W_i(K)$. The respective homogeneity properties then imply that  for all $\alpha \in \mathbb{R}$,
$$
\alpha^{n \frac{n-1}{n+1}}IS_1(K) = \sum_{i=0}^n \lambda_i  \alpha^{n-i} W_i(K), 
$$
and in particular, for $K=B^n_2$, that for all $\alpha \in \mathbb{R}$, 
\begin{equation}\label{step1}
n\  \alpha^{n \frac{n-1}{n+1}} = \sum_{i=0}^n \lambda_i  \alpha^{n-i} = \lambda_0 \alpha^n + \lambda_1 \alpha^{n-1} + \cdots + \lambda_n.
\end{equation}
Letting $\alpha = 0$ in (\ref{step1}) shows that $\lambda_n=0$. This means that for all $\alpha \in \mathbb{R}$,
\begin{equation*}
n\  \alpha^{n \frac{n-1}{n+1}} = \sum_{i=0}^{n-1} \lambda_i  \alpha^{n-i} = \lambda_0 \alpha^n + \lambda_1 \alpha^{n-1} + \cdots + \lambda_{n-1} \alpha.
\end{equation*}
Differentiation gives
\begin{equation}\label{step2}
n \left(n \frac{n-1}{n+1}\right)  \alpha^{n \frac{n-1}{n+1}-1} = n  \lambda_0 \alpha^{n-1} + (n-1) \lambda_1 \alpha^{n-2} + \cdots + \lambda_{n-1}.
\end{equation}
Letting $\alpha = 0$ in (\ref{step2}) shows that $\lambda_{n-1}=0$.
We continue differentiating till  the largest $k \in \mathbb{N}$ for which the exponent  $n \frac{n-1}{n+1}-k $ of $\alpha$ on the left hand side of the equality is strictly larger than $0$.
We can take $k=n-2$ and  get that $\lambda_n=\lambda_{n-1}= \cdots =\lambda_2=0$. Thus  equality (\ref{step1}) reduces to the following: there exist $\lambda_0$ and $\lambda_1$
such that for all $\alpha \in \mathbb{R}$,
$$
\frac{n}{ \alpha ^\frac{n-1}{n+1}} = \lambda_0 \alpha + \lambda_1, 
$$
which is not possible. The proof is therefore complete.
\end{proof}
\noindent
\vskip 4mm
\noindent
{\bf Acknowledgement}:  \   We would like to thank the referee for all the helpful comments.

\vskip 4mm
\noindent
Ohad Giladi\\
{\small School of Mathematical and Physical Sciences}\\
{\small University of Newcastle}\\
{\small  Callaghan, NSW 2308, Australia}\\
{\small \tt ohad.giladi@newcastle.edu.au}
\vskip 3mm
\noindent
{\small Han Huang}\\
{\small Department of Mathematics}\\
{\small  Georgia Institute of Technology, Atlanta, Georgia}\\
{\small  \tt hhuang421@math.gatech.edu} 
\vskip 3mm
\noindent
Carsten Sch\"utt\\
{\small Mathematisches  Seminar}\\
{\small Universit\"at Kiel}\\
{\small Germany }\\
{\small \tt schuett@math.uni-kiel.de}
\vskip 3mm
\noindent
Elisabeth M. Werner\\
{\small Department of Mathematics \ \ \ \ \ \ \ \ \ \ \ \ \ \ \ \ \ \ \ Universit\'{e} de Lille 1}\\
{\small Case Western Reserve University \ \ \ \ \ \ \ \ \ \ \ \ \ UFR de Math\'{e}matique }\\
{\small Cleveland, Ohio 44106, U. S. A. \ \ \ \ \ \ \ \ \ \ \ \ \ \ \ 59655 Villeneuve d'Ascq, France}\\
{\small \tt elisabeth.werner@case.edu}\\ \\

\end{document}